\documentclass[a4paper,oneside,11pt]{article}%
\usepackage{makeidx}
\usepackage[english]{babel}
\usepackage{amsmath}
\usepackage{amsfonts}
\usepackage{amssymb}
\usepackage{stmaryrd}
\usepackage{graphicx}
\usepackage{mathrsfs}
\usepackage{cite}
\usepackage[colorlinks,linkcolor=red,anchorcolor=blue,citecolor=blue,urlcolor=blue]{hyperref}
\usepackage[symbol*,ragged]{footmisc}
\providecommand{\U}[1]{\protect\rule{.1in}{.1in}}
\providecommand{\U}[1]{\protect \rule{.1in}{.1in}}

\newtheorem{theorem}{Theorem}[section]

\newtheorem{lemma}[theorem]{Lemma}

\newenvironment{proof}[1][Proof]{\noindent \textbf{#1.} }{\  \rule{0.5em}{0.5em}}
\numberwithin{equation}{section}

\begin{document}

\title{On the Sum of Divisors of Mixed Powers }

\author{Jinjiang Li\footnotemark[1]\,\,\,\, \, \& \,\, Min Zhang\footnotemark[2]    \vspace*{-4mm} \\
$\textrm{\small Department of Mathematics, China University of Mining and Technology}^{*\,\dag}$
                    \vspace*{-4mm} \\
     \small  Beijing 100083, P. R. China  }

\footnotetext[2]{Corresponding author. \\
    \quad\,\, \textit{ E-mail addresses}: \href{mailto:jinjiang.li.math@gmail.com}{jinjiang.li.math@gmail.com} (J. Li), \href{mailto:min.zhang.math@gmail.com}{min.zhang.math@gmail.com} (M. Zhang).   }

\date{}
\maketitle


{\textbf{Abstract}}: Let~$d(n)$~denote the Dirichlet divisor function. Define
\begin{equation*}
  \mathcal{S}_{k}(x)=\sum_{\substack{1\leqslant n_1,n_2,n_3 \leqslant x^{1/2} \\ 1\leqslant n_4\leqslant x^{1/k} }} d(n_1^2+n_2^2+n_3^2+n_4^k),
  \qquad 3\leqslant k\in \mathbb{N}.
\end{equation*}
In this paper, we establish an asymptotic formula of~$\mathcal{S}_k(x)$~and prove that
\begin{equation*}
   \mathcal{S}_k(x)=C_1(k)x^{3/2+1/k}\log x+C_2(k)x^{3/2+1/k}+O(x^{3/2+1/k-\delta_k+\varepsilon}),
\end{equation*}
where~$C_1(k),\,C_2(k)$~are two constants depending only on~$k,$~with~$\delta_3=\frac{19}{60},\,\delta_4=\frac{5}{24},\,\delta_5=\frac{19}{140},\,\delta_6=\frac{25}{192},\,
\delta_7=\frac{457}{4032},\,\delta_k=\frac{1}{k+2}+\frac{1}{2k^2(k-1)}$~for~$k\geqslant8.$~

{\textbf{Keywords:}} Divisor function; circle method; mixed power; asymptotic formula

\textbf{Mathematics Subject Classification 2010:} 11P05,\,11P32,\,11P55

\section{Introduction and main result}

   Let~$d(n)$~be the Dirichlet divisor function. Gafurov~\cite{Gafurov-1,Gafurov-2} studied the number of
divisors of a binary quadratic form and obtained the asymptotic formula
\begin{equation*}
   \sum_{1\leqslant m,n\leqslant x}d(m^2+n^2)=A_1x^2\log x+A_2x^2+O(x^{5/3}\log^9x),
\end{equation*}
where $A_1$ and $A_2$ are certain constants. Later this result was improved by Yu \cite{Yu} ,who gave the asymptotic formula
\begin{equation*}
   \sum_{1\leqslant m,n\leqslant x}d(m^2+n^2)=A_1x^2\log x+A_2x^2+O(x^{3/2+\varepsilon}),
\end{equation*}
for any fixed $\varepsilon>0.$

   In 2000, Calder\'{a}n and de Velasco~\cite{Calderon-Velasco} studied the divisors of the quadratic form~$m_1^2+m_2^2+m_3^2$ and established
   the asymptotic formula
\begin{equation}\label{Calderon-Velasco}
   \sum_{1\leqslant m_1,m_2,m_3\leqslant x}d(m_1^2+m_2^2+m_3^2)=\frac{8\zeta(3)}{5\zeta(5)}x^3\log x+O(x^{3}).
\end{equation}
In 2012, Guo and Zhai~\cite{Guo-Zhai} improved (\ref{Calderon-Velasco}) to
\begin{equation*}
   \sum_{1\leqslant m_1,m_2,m_3\leqslant x}d(m_1^2+m_2^2+m_3^2)=2C_1I_1x^3\log x+(C_1I_2+C_2I_1)x^3+O(x^{8/3+\varepsilon}),
\end{equation*}
where
\begin{equation*}
    C_1=\sum_{q=1}^\infty \frac{1}{q^4} \sum_{\substack{a=1\\(a,q)=1}}^q\bigg(\sum_{h=1}^qe\Big(\frac{ah^2}{q}\Big)\bigg)^3,
\end{equation*}
\begin{equation*}
    C_2=\sum_{q=1}^\infty \frac{-2\log q+2\gamma}{q^4} \sum_{\substack{a=1\\(a,q)=1}}^q\bigg(\sum_{h=1}^qe\Big(\frac{ah^2}{q}\Big)\bigg)^3,
\end{equation*}
\begin{equation*}
    I_1=\int_{-\infty}^{+\infty}\Big(\int_0^1e(\mu^2\beta)\mathrm{d}\mu\Big)^3\Big(\int_0^3e(-\mu\beta)\mathrm{d}\mu\Big)\mathrm{d}\beta
\end{equation*}
 and
\begin{equation*}
    I_2=\int_{-\infty}^{+\infty}\Big(\int_0^1e(\mu^2\beta)\mathrm{d}\mu\Big)^3\Big(\int_0^3e(-\mu\beta)\log\mu\mathrm{d}\mu\Big)\mathrm{d}\beta.
\end{equation*}
Later, Zhao \cite{Zhao} improved the error term $O(x^{8/3+\varepsilon})$ to $O(x^2\log^7x).$ Moreover, L\"{u} and Mu~\cite{Lv-Mu} consider the nonhomogeneous
case. They proved that, for~$k\geqslant3,$~there holds
\begin{equation*}
   \sum_{\substack{1\leqslant n_1,n_2\leqslant x^{1/2}\\1\leqslant n_3\leqslant x^{1/k}}}d(n_1^2+n_2^2+n_3^k)=
   A(k)x^{1+1/k}\log x+B(k)x^{1+1/k}+O(x^{1+1/k-\theta(k)+\varepsilon}),
\end{equation*}
where~$A(k),\,B(k)$~are two constants depending only on $k$, $\theta(3)=5/42,\,\theta(4)=1/16,\,\theta(5)=1/40,\,\theta(k)=1/(k2^{k-1})$ for $6\leqslant k\leqslant 7$ and $\theta(k)=1/({2k^2(k-1)})$ for $k\geqslant8$.

   In 2014, Hu~\cite{Hu} considered the divisors of the quaternary form $m_1^2+m_2^2+m_3^2+m_4^2$ and obtained
\begin{eqnarray}
   &   & \sum_{1\leqslant m_1,m_2,m_3,m_4\leqslant x}d(m_1^2+m_2^2+m_3^2+m_4^2)   \nonumber   \\
   &   &   \qquad\qquad= 2\mathcal{C}_1\mathcal{I}_1x^4\log x+(\mathcal{C}_1\mathcal{I}_2+\mathcal{C}_2\mathcal{I}_1)x^4+O(x^{7/2+\varepsilon}),
\end{eqnarray}
where
\begin{equation*}
   \mathcal{C}_1=\sum_{q=1}^{\infty}\frac{1}{q^5}\sum_{\substack{a=1\\(a,q)=1}}^{q}\bigg(\sum_{r=1}^qe\Big(\frac{ar^2}{q}\Big)\bigg)^4,
\end{equation*}
\begin{equation*}
   \mathcal{C}_2=\sum_{q=1}^{\infty}\frac{-2\log q+2\gamma}{q^5}\sum_{\substack{a=1\\(a,q)=1}}^{q}\bigg(\sum_{r=1}^qe\Big(\frac{ar^2}{q}\Big)\bigg)^4,
\end{equation*}
\begin{equation*}
  \mathcal{I}_1=\int_{-\infty}^{+\infty}\bigg(\int_0^1e(u^2\lambda)\mathrm{d}u\bigg)^4\bigg(\int_0^4e(-u\lambda)\mathrm{d}u\bigg)\mathrm{d}\lambda
\end{equation*}
and
\begin{equation*}
    \mathcal{I}_2=\int_{-\infty}^{+\infty}\bigg(\int_0^1e(u^2\lambda)\mathrm{d}u\bigg)^4\bigg(\int_0^4e(-u\lambda)\log u\mathrm{d}u\bigg)\mathrm{d}\lambda.
\end{equation*}
Later, Liu and Hu~\cite{Liu-Hu} improved the error term $O\big(x^{7/2+\varepsilon}\big)$ to $O(x^3\log x)$.

  In this paper, we consider the nonhomogeneous case of the form $n_1^2+n_2^2+n_3^2+n_4^k$ and establish the following theorem.

\begin{theorem}\label{mixed-power-theorem}
   Let
  \begin{equation*}
     S_{k}(x)=\sum_{\substack{1\leqslant n_1,n_2,n_3 \leqslant x^{1/2} \\ 1\leqslant n_4\leqslant x^{1/k} }} d(n_1^2+n_2^2+n_3^2+n_4^k),
     \qquad 3\leqslant k\in \mathbb{N}.
  \end{equation*}
  Then we have
  \begin{equation*}
     \mathcal{S}_k(x)=\mathfrak{S}_1\mathfrak{J}_1x^{3/2+1/k}\log x+(\mathfrak{S}_1\mathfrak{J}_2+\mathfrak{S}_2\mathfrak{J}_1)x^{3/2+1/k}+O(x^{3/2+1/k-\delta_k+\varepsilon}),
  \end{equation*}
where
\begin{eqnarray*}
  \delta_3=\frac{19}{60},\qquad\delta_4=\frac{5}{24},\qquad\delta_5=\frac{19}{140},\qquad\delta_6=\frac{25}{192},\,  \\
  \delta_7=\frac{457}{4032},\qquad\delta_k=\frac{1}{k+2}+\frac{1}{2k^2(k-1)} \textrm{\quad for\quad} k\geqslant 8,
\end{eqnarray*}
\begin{equation*}
   \mathfrak{S}_1=\sum_{q=1}^{\infty}\frac{1}{q^5}\sum_{ \substack{a=1\\(a,q)=1} }^q
          \bigg(\sum_{r=1}^{q}e\left(\frac{ar^2}{q}\right)\bigg)^3
          \bigg(\sum_{r=1}^{q}e\left(\frac{ar^k}{q}\right)\bigg),
\end{equation*}
\begin{equation*}
    \mathfrak{S}_2=\sum_{q=1}^{\infty}\frac{-2\log q+2\gamma}{q^5}\sum_{ \substack{a=1\\(a,q)=1} }^q
          \bigg(\sum_{r=1}^{q}e\left(\frac{ar^2}{q}\right)\bigg)^3
          \bigg(\sum_{r=1}^{q}e\left(\frac{ar^k}{q}\right)\bigg),
\end{equation*}
\begin{equation*}
    \mathfrak{J}_1=\int_{-\infty}^{+\infty}
        \bigg(\int_0^1e(\alpha\mu^2)\mathrm{d}\mu\bigg)^3 \bigg(\int_0^1e(\alpha\mu^k)\mathrm{d}\mu\bigg) \bigg(\int_0^3e(-\alpha\mu)\mathrm{d}\mu\bigg)
    \mathrm{d}\alpha,
\end{equation*}
\begin{equation*}
    \mathfrak{J}_2=\int_{-\infty}^{+\infty}
        \bigg(\int_0^1e(\alpha\mu^2)\mathrm{d}\mu\bigg)^3 \bigg(\int_0^1e(\alpha\mu^k)\mathrm{d}\mu\bigg)
         \bigg(\int_0^3e(-\alpha\mu)\log\mu\mathrm{d}\mu\bigg)
          \mathrm{d}\alpha.
\end{equation*}
\end{theorem}

\textbf{Notation.} Throughout this paper,~$x$~always denotes a sufficiently large real number;
$\varepsilon$~always denotes an arbitrary small positive constant, which may not be the same at different occurrences. $e(x)=e^{2\pi ix};$~
 $f(x)\ll g(x)$ means that~$f=O(g(x)).$ For the sake of brevity, we define
\begin{equation*}
   \mathfrak{J}_1(\beta)= \bigg(\int_0^1e(\beta\mu^2)\mathrm{d}\mu\bigg)^3 \bigg(\int_0^1e(\beta\mu^k)\mathrm{d}\mu\bigg) \bigg(\int_0^3e(-\beta\mu)\mathrm{d}\mu\bigg)
\end{equation*}
and
\begin{equation*}
   \mathfrak{J}_2(\beta)= \bigg(\int_0^1e(\beta\mu^2)\mathrm{d}\mu\bigg)^3 \bigg(\int_0^1e(\beta\mu^k)\mathrm{d}\mu\bigg) \bigg(\int_0^3e(-\beta\mu)\log\mu\mathrm{d}\mu\bigg).
\end{equation*}

\section{Preliminary Lemmas }

 For any~$\alpha\in\mathbb{R},$~ define
\begin{equation*}
  f_{\ell}(\alpha)=\sum_{1\leqslant n\leqslant x^{1/\ell}}e(\alpha n^\ell),\qquad f(\alpha)=\sum_{1\leqslant n\leqslant4x}d(n)e(\alpha n).
\end{equation*}

\begin{lemma}
For any real numbers~$\alpha$~and~$\tau\geqslant1,$~there exist integers~$a$~and~$q,\,(a,q)=1,\,1\leqslant q\leqslant\tau,$~ such that
  \begin{equation*}
      \alpha=\frac{a}{q}+\lambda, \qquad |\lambda|\leqslant\frac{1}{q\tau}.
  \end{equation*}
\end{lemma}
\begin{proof}
   See C. D. Pan and C. B. Pan~\cite{Pan-Pan}, Lemma~5.19.
\end{proof}

Let
\begin{equation}\label{parameter}
  \log x<2Q<\tau<x,\quad Q\tau\asymp x,\quad  Q\ll x^{2/(k+2)}.
\end{equation}
For any~$1\leqslant a<q\leqslant Q$~with~$(a,q)=1,$~define
\begin{equation*}
    \mathfrak{M}(a,q)=\bigg[\frac{a}{q}-\frac{1}{q\tau},\frac{a}{q}+\frac{1}{q\tau}\bigg]
\end{equation*}
and
\begin{equation*}
   \mathfrak{M}=\bigcup_{q\leqslant Q}\bigcup_{\substack{1\leqslant a\leqslant q\\(a,q)=1}}\mathfrak{M}(a,q),
   \qquad  \mathfrak{m}=\bigg[\frac{1}{\tau},1+\frac{1}{\tau}\bigg]\setminus\mathfrak{M}.
\end{equation*}
We call~$\mathfrak{M}$~the major arc and~$\mathfrak{m}$~the minor arc. By the definition of~$\mathcal{S}_k(x)$~and orthogonality of
exponential function, we have
\begin{eqnarray*}
   \mathcal{S}_k(x) & = & \int_0^1 f_2^3(\alpha)f_k(\alpha)f(-\alpha)\mathrm{d}\alpha \\
                    & = & \bigg\{\int_{\mathfrak{M}}+\int_{\mathfrak{m}}\bigg\}f_2^3(\alpha)f_k(\alpha)f(-\alpha)\mathrm{d}\alpha \\
                    & =: &  \mathcal{S}_k(x,\mathfrak{M})+ \mathcal{S}_k(x,\mathfrak{m}).
\end{eqnarray*}

\begin{lemma}\label{S_k-estimate}
   For any~$a,q\in\mathbb{Z}$~with~$(a,q)=1$~and~$q>0,$~let
   \begin{equation*}
      S_k(q,a)=\sum_{r=1}^{q}e\left(\frac{ar^k}{q}\right).
   \end{equation*}
Then we have
\begin{equation*}
   S_{k}(q,a)\ll q^{(k-1)/k}.
\end{equation*}
\end{lemma}
\begin{proof}
   See Vaughan~\cite{Vaughan}, Theorem~4.2.
\end{proof}

\begin{lemma}\label{Mu-estimate}
   For~$k\geqslant1,$~we have
   \begin{equation}
      \int_0^1e(\beta\mu^k)\mathrm{d}\mu\ll\frac{1}{(1+|\beta|)^{1/k}},
   \end{equation}
   \begin{equation}
     \int_0^3e(\beta\mu^k)\log\mu\mathrm{d}\mu\ll\frac{\log(2+|\beta|)}{1+|\beta|}.
   \end{equation}
\end{lemma}
\begin{proof}
   See L\"{u} and Mu~\cite{Lv-Mu}, Lemma~1.2.
\end{proof}

\begin{lemma}\label{fk-asymp}
   Suppose that~$(a,q)=1$~and~$\alpha=\frac{a}{q}+\beta.$~Then
   \begin{equation*}
      f_k(\alpha)=V_k(\alpha,q,a)+O\big( q^{1/2+\varepsilon}(1+x|\beta|)^{1/2}\big).
   \end{equation*}
   Moreover, if~$|\beta|\leqslant\frac{x^{1/k-1}}{2kq},$~then
   \begin{equation*}
      f_k(\alpha)=V_k(\alpha,q,a)+O\big( q^{1/2+\varepsilon}\big),
   \end{equation*}
   where
   \begin{equation*}
       V_k(\alpha,q,a)=x^{1/k}\frac{S_k(q,a)}{q} \int_0^1e(x\beta\mu^k)\mathrm{d}\mu.
   \end{equation*}
\end{lemma}
\begin{proof}
   See Vaughan~\cite{Vaughan}, Theorem~4.1.
\end{proof}

\begin{lemma}\label{zhai-lemma}
   Suppose that~$\alpha=\frac{a}{q}+\beta\in\mathfrak{M}$~and~$Q\tau\leqslant x,\, \tau>x^{1/2+\varepsilon}.$~Then
   \begin{eqnarray}
     f(-\alpha)\! & = & \!\frac{x\log x}{q} \int_0^3e(-\mu x\beta)\mathrm{d}\mu+\frac{x}{q}\int_0^3e(-\mu x\beta)\log\mu\mathrm{d}\mu   \nonumber \\
                  &   & +\frac{-2\log q+2\gamma}{q}x\int_0^3e(-\mu x\beta)\mathrm{d}\mu+O(\Delta),    \nonumber
   \end{eqnarray}
   where
   \begin{equation}\label{Delta}
       \Delta=x^{\varepsilon}\big(q^{1/2}x\tau^{-1} +q^{2/3}x^{1/3} \big).
   \end{equation}
\end{lemma}
\begin{proof}
   See Guo and Zhai~\cite{Guo-Zhai}, Lemma~7.1.
\end{proof}

\begin{lemma}\label{compare}
   Suppose that
   \begin{equation*}
      L(H)=\sum_{i=1}^mA_iH^{a_i}+\sum_{j=1}^nB_jH^{b_j},
   \end{equation*}
   where~$A_i,\,B_j,\,a_i$~and~$b_j$~are positive. Assume that~$H_1\leqslant H_2.$~Then there exists some~$\mathscr{H}$~with
   $H_1\leqslant\mathscr{H}\leqslant H_2$~and
   \begin{equation*}
     L(\mathscr{H})\ll \sum_{i=1}^mA_iH_1^{a_i}+\sum_{j=1}^nB_jH_2^{b_j}+\sum_{i=1}^m\sum_{j=1}^n\big(A_i^{b_j}B_j^{a_i}\big)^{1/(a_i+b_j)}.
   \end{equation*}
   The implied constant depends only on~$m$~and~$n.$~
\end{lemma}
\begin{proof}
   See Srinivasan~\cite{Srinivasan}, Lemma 3 or Graham and Kolesnik~\cite{Graham-Kolesnik}, Lemma~2.4.
\end{proof}

\begin{lemma}\label{Weyl-inequality}
   Suppose that
   \begin{equation*}
      (a,q)=1, \quad \left|\alpha-\frac{a}{q}\right|\leqslant\frac{1}{q^2},\quad \phi(x)=\alpha x^k+\alpha_1x^{k-1}+\cdots+\alpha_{k-1}x+\alpha_k.
   \end{equation*}
   Then
   \begin{equation*}
       \sum_{1\leqslant x\leqslant Y}e(\phi(x))\ll Y^{1+\varepsilon}\big(q^{-1}+Y^{-1}+qY^{-k} \big)^{1/2^{k-1}}.
   \end{equation*}
\end{lemma}
\begin{proof}
   See Vaughan~\cite{Vaughan}, Lemma~2.4.
\end{proof}

\begin{lemma}\label{Hua-inequality}
   Suppose that~$1\leqslant j\leqslant k.$~Then
   Then
   \begin{equation*}
       \int_{0}^1 \bigg| \sum_{m=1}^Ye\big(\alpha m^k\big) \bigg|^{2^j} \mathrm{d}\alpha \ll Y^{2^{j}-j+\varepsilon}.
   \end{equation*}
\end{lemma}
\begin{proof}
   See Vaughan~\cite{Vaughan}, Lemma~2.5.
\end{proof}

\begin{lemma}\label{wooley-estimate}
  Let~$j$~be an integer with~$j\geqslant2.$~Suppose that there exist integers~$a,q$~with~$q\geqslant1,\,(a,q)=1$~such that~$|\alpha-\frac{a}{q}|\leqslant\frac{1}{q^2}$~and~$q\leqslant x.$~Then one has
  \begin{equation*}
     f_j(\alpha)\ll x^{1/j+\varepsilon}\big(q^{-1}+x^{-1/j}+qx^{-1}\big)^{1/2j(j-1)}.
  \end{equation*}
\end{lemma}
\begin{proof}
   See Wooley~\cite{Wooley}, Theorm~1.5.
\end{proof}

\section{Proof of Theorem~\ref{mixed-power-theorem}}

It is obvious that
\begin{equation*}
   \mathcal{S}_k(x,\mathfrak{M})=\sum_{1\leqslant q\leqslant Q}\sum_{\substack{a=1\\(a,q)=1}}^q
   \int_{\frac{a}{q}-\frac{1}{q\tau}}^{\frac{a}{q}+\frac{1}{q\tau}}
   f_2^3(\alpha)f_k(\alpha)f(-\alpha)\mathrm{d}\alpha.
\end{equation*}

For~$\alpha=\frac{a}{q}+\beta\in\mathfrak{M},$~by Lemma~\ref{fk-asymp}~and Lemma~\ref{zhai-lemma}, we have
\begin{eqnarray*}
  &    & f_2^3(\alpha)f_k(\alpha)f(-\alpha)  \\
   & = & \Big(V_2(\alpha,q,a)+O(q^{1/2+\varepsilon})\Big)^3 \Big(V_k(\alpha,q,a)+O\big(q^{1/2+\varepsilon}(1+x|\beta|)^{1/2}\big)\Big) \\
   &   & \times\bigg(\frac{x\log x}{q}\int_0^3e(-\mu x\beta)\mathrm{d}\mu+\frac{x}{q}\int_0^3e(-\mu x\beta)\log\mu\mathrm{d}\mu     \\
   &   & \quad\,\,\,+\frac{-2\log q+2\gamma}{q}x\int_0^3e(-\mu x\beta)\mathrm{d}\mu+O(\Delta)\bigg) \\
\end{eqnarray*}
\begin{eqnarray*}
   & = & x^{3/2+1/k}\frac{S_2^3(q,a)S_k(q,a)}{q^4}\bigg(\int_0^1e(x\beta\mu^2)\mathrm{d}\mu\bigg)^3\bigg(\int_0^1e(x\beta\mu^k)\mathrm{d}\mu\bigg)  \\
   &   & \times\bigg(\frac{x\log x}{q}\int_0^3e(-\mu x\beta)\mathrm{d}\mu+\frac{x}{q}\int_{0}^3e(-\mu x\beta)\log\mu\mathrm{d}\mu  \\
   &   &  \quad\,\,\,+\frac{-2\log q+2\gamma}{q}x\int_0^3e(-\mu x\beta)\mathrm{d}\mu\bigg) +O\big(x^{5/2+\varepsilon}q^{-2}(1+x|\beta|)^{-2}\big)  \\
   &   &   +O\big(x^{1+1/k+\varepsilon}q^{1/2-1/k}(1+x|\beta|)^{-1-1/k}+x^{1+\varepsilon}q(1+x|\beta|)^{-1/2}\big)     \\
   &   &   +O\Big(\Delta\big(x^{3/2+1/k}q^{-3/2-1/k}(1+x|\beta|)^{-3/2-1/k}+q^2x^{\varepsilon}(1+x|\beta|)^{1/2}    \\
   &   &   \qquad\quad\quad +x^{1/k+\varepsilon}q^{3/2-1/k}(1+x|\beta|)^{-1/k}+x^{3/2+\varepsilon}q^{-1}(1+x|\beta|)^{-1}      \big)\Big),
\end{eqnarray*}
where~$\Delta$~is defined in~(\ref{Delta}). Then we have
\begin{eqnarray}\label{single-major-arc}
  &    &   \int_{\frac{a}{q}-\frac{1}{q\tau}}^{\frac{a}{q}+\frac{1}{q\tau}}f_2^3(\alpha)f_k(\alpha)f(-\alpha)\mathrm{d}\alpha   \nonumber  \\
  &  = &   \int_{-\frac{1}{q\tau}}^{\frac{1}{q\tau}}f_2^3\Big(\frac{a}{q}+\beta  \Big) f_k\Big(\frac{a}{q}+\beta  \Big)
           f\Big(-\frac{a}{q}-\beta  \Big)\mathrm{d}\beta   \nonumber \\
  &  =  &  \frac{S_2^3(q,a)S_k(q,a)}{q^5}x^{5/2+1/k}\log x   \nonumber  \\
  &     &  \qquad \times\int_{-\frac{1}{q\tau}}^{\frac{1}{q\tau}}\bigg(\int_{0}^{1}e(x\beta\mu^2)\mathrm{d}\mu\bigg)^3
           \bigg(\int_{0}^{1}e(x\beta\mu^k)\mathrm{d}\mu\bigg)\bigg(\int_{0}^{3}e(x\beta\mu)\mathrm{d}\mu\bigg)\mathrm{d}\beta   \nonumber  \\
  &     &  + \frac{S_2^3(q,a)S_k(q,a)}{q^5}x^{5/2+1/k}  \nonumber  \\
  &     &  \qquad\times\int_{-\frac{1}{q\tau}}^{\frac{1}{q\tau}}\bigg(\int_{0}^{1}e(x\beta\mu^2)\mathrm{d}\mu\bigg)^3
           \bigg(\int_{0}^{1}e(x\beta\mu^k)\mathrm{d}\mu\bigg)\bigg(\int_{0}^{3}e(x\beta\mu)\log\mu\mathrm{d}\mu\bigg)\mathrm{d}\beta   \nonumber  \\
  &     &  +(-2\log q+2\gamma) \frac{S_2^3(q,a)S_k(q,a)}{q^5}x^{5/2+1/k}  \nonumber   \\
  &     &  \qquad \times\int_{-\frac{1}{q\tau}}^{\frac{1}{q\tau}}\bigg(\int_{0}^{1}e(x\beta\mu^2)\mathrm{d}\mu\bigg)^3
           \bigg(\int_{0}^{1}e(x\beta\mu^k)\mathrm{d}\mu\bigg)\bigg(\int_{0}^{3}e(x\beta\mu)\mathrm{d}\mu\bigg)\mathrm{d}\beta    \nonumber  \\
  &     &  +O\big(x^{3/2+\varepsilon}q^{-2}+x^{1/k+\varepsilon}q^{1/2-1/k}+x^{3/2+1/k+\varepsilon}q^{-1-1/k}\tau^{-1}   \nonumber  \\
  &     &  \qquad\quad+x^{5/6+1/k+\varepsilon}q^{-5/6-1/k}+x^{3/2+\varepsilon}q^{-1/2}\tau^{-1}+x^{5/6+\varepsilon}q^{-1/3}\big)  \nonumber   \\
  &  =  &  \frac{S_2^3(q,a)S_k(q,a)}{q^5}x^{3/2+1/k}\log x \int_{-\frac{x}{q\tau}}^{\frac{x}{q\tau}}\mathfrak{J}_1(\beta)\mathrm{d}\beta  \nonumber  \\
  &     &  +\frac{S_2^3(q,a)S_k(q,a)}{q^5}x^{3/2+1/k} \int_{-\frac{x}{q\tau}}^{\frac{x}{q\tau}}\mathfrak{J}_2(\beta)\mathrm{d}\beta \nonumber \\
  &     &  +(-2\log q+2\gamma) \frac{S_2^3(q,a)S_k(q,a)}{q^5}x^{3/2+1/k}
           \int_{-\frac{x}{q\tau}}^{\frac{x}{q\tau}}\mathfrak{J}_1(\beta)\mathrm{d}\beta \nonumber \\
  &     &  +O\big(x^{3/2+\varepsilon}q^{-2}+x^{1/k+\varepsilon}q^{1/2-1/k}+x^{3/2+1/k+\varepsilon}q^{-1-1/k}\tau^{-1}  \nonumber  \\
  &     &  \qquad\quad+x^{5/6+1/k+\varepsilon}q^{-5/6-1/k}+x^{3/2+\varepsilon}q^{-1/2}\tau^{-1}+x^{5/6+\varepsilon}q^{-1/3}\big).
\end{eqnarray}
Applying Lemma~\ref{Mu-estimate}, we have
\begin{equation*}
   \mathfrak{J}_1(\beta)\ll\frac{1}{(1+|\beta|)^{5/2+1/k}},\qquad \mathfrak{J}_2(\beta)\ll\frac{\log(2+|\beta|)}{(1+|\beta|)^{5/2+1/k}}.
\end{equation*}
Therefore, we have
\begin{equation}\label{J_1-tail}
   \int_{|\beta|>\frac{x}{q\tau}}\mathfrak{J}_1(\beta)\mathrm{d}\beta\ll\int_{|\beta|>\frac{x}{q\tau}}\frac{1}{(1+|\beta|)^{5/2+1/k}}\mathrm{d}\beta
   \ll\frac{1}{\Big(1+\displaystyle\frac{x}{q\tau}\Big)^{3/2+1/k}}
\end{equation}
and
\begin{equation}\label{J_2-tail}
  \int_{|\beta|>\frac{x}{q\tau}}\mathfrak{J}_2(\beta)\mathrm{d}\beta\ll\int_{|\beta|>\frac{x}{q\tau}}\frac{\log(2+|\beta|)}{(1+|\beta|)^{5/2+1/k}}
  \mathrm{d}\beta \ll \frac{\log\Big(2+\displaystyle\frac{x}{q\tau}\Big)}{\Big(1+\displaystyle\frac{x}{q\tau}\Big)^{3/2+1/k}}.
\end{equation}
Hence, from (\ref{single-major-arc}),~(\ref{J_1-tail}) and (\ref{J_2-tail}), we get
\begin{eqnarray}\label{Single-major-1}
   &   &  \int_{\frac{a}{q}-\frac{1}{q\tau}}^{\frac{a}{q}+\frac{1}{q\tau}} f_2^3(\alpha)f_k(\alpha)f(-\alpha)\mathrm{d}\alpha   \nonumber    \\
   & = &  \mathfrak{J}_1\frac{S_2^3(q,a)S_k(q,a)}{q^5}x^{3/2+1/k}\log x+\mathfrak{J}_2\frac{S_2^3(q,a)S_k(q,a)}{q^5}x^{3/2+1/k} \nonumber  \\
   &   &  +(-2\log q+2\gamma)\mathfrak{J}_1\frac{S_2^3(q,a)S_k(q,a)}{q^5} x^{3/2+1/k}  \nonumber  \\
   &   &  +O\big(x^\varepsilon q^{-1}\tau^{3/2+1/k}+x^{3/2+\varepsilon}q^{-2}
                 +x^{1/k+\varepsilon}q^{1/2-1/k} +x^{3/2+1/k+\varepsilon}q^{-1-1/k}\tau^{-1}  \nonumber  \\
   &   &  +x^{5/6+1/k+\varepsilon}q^{-5/6-1/k} +x^{3/2+\varepsilon}q^{-1/2}\tau^{-1}+x^{5/6+\varepsilon}q^{-1/3}\big).
\end{eqnarray}
By Lemma~\ref{S_k-estimate}, it follows that
\begin{equation}\label{singular-series-tail}
  \sum_{q>Q}\sum_{\substack{a=1\\(a,q)=1}}^q\frac{S_2^3(q,a)S_k(q,a)}{q^5}\ll\sum_{q>Q}q^{-3/2-1/k}\ll Q^{-1/2-1/k}.
\end{equation}
Therefore, from (\ref{Single-major-1}) and (\ref{singular-series-tail}), we obtain
\begin{eqnarray}\label{major-arc-last}
       \mathcal{S}_k(x,\mathfrak{M})
   & = & \sum_{1\leqslant q\leqslant Q}\sum_{\substack{a=1\\(a,q)=1}}^q \int_{\frac{a}{q}-\frac{1}{q\tau}}^{\frac{a}{q}+\frac{1}{q\tau}}
                f_2^3(\alpha)f_k(\alpha)f(-\alpha)\mathrm{d}\alpha  \nonumber    \\
   & = & \mathfrak{S}_1\mathfrak{J}_1x^{3/2+1/k}\log x+(\mathfrak{S}_1\mathfrak{J}_2+\mathfrak{S}_2\mathfrak{J}_1)x^{3/2+1/k}    \nonumber    \\
   &   & +O\big(   x^{3/2+1/k+\varepsilon}Q^{-1/2-1/k}+x^\varepsilon Q\tau^{3/2+1/k}+x^{1/k+\varepsilon}Q^{5/2-1/k}
                                     \nonumber    \\
   &   &            \qquad  +x^{3/2+1/k+\varepsilon}Q^{1-1/k}\tau^{-1}+x^{5/6+1/k+\varepsilon}Q^{7/6-1/k}    \nonumber    \\
   &   &            \qquad +x^{3/2+\varepsilon}Q^{3/2}\tau^{-1}+x^{5/6+\varepsilon}Q^{5/3}  \big).
\end{eqnarray}

  It remains to estimate the integral on the minor arc~$\mathfrak{m}.$~At this time, for~$\alpha\in\mathfrak{m},$~we have~$Q<q\leqslant\tau.$~We consider
  four different cases as follows.

  \textbf{Case 1.} If~$3\leqslant k\leqslant5,$~by noting the fact that~$Q\tau\asymp x,\, Q<x^{2/(k+2)}$~and Lemma~\ref{Weyl-inequality}, we have
\begin{eqnarray}
   f_2(\alpha) & \ll & x^{1/2+\varepsilon}\big(q^{-1}+x^{-1/2}+qx^{-1}\big)^{1/2}  \nonumber  \\
               & \ll & x^{1/2+\varepsilon}\big(Q^{-1/2}+x^{-1/4}+\tau^{1/2}x^{-1/2}\big)     \nonumber  \\
               & \ll & x^{1/2+\varepsilon}Q^{-1/2}.
\end{eqnarray}
Also, we have
\begin{equation}\label{f-a}
   \int_0^1|f(-\alpha)|^2\mathrm{d}\alpha=\sum_{n\leqslant4x}d^2(n)\ll x\log^3x.
\end{equation}
Therefore, it follows from H\"{o}lder's inequality, Lemma~\ref{Hua-inequality} and (\ref{f-a}) that
\begin{eqnarray}
     &     &  \mathcal{S}_k(x,\mathfrak{m})  =  \int_{\mathfrak{m}}f_2^3(\alpha)f_{k}(\alpha)f(-\alpha)\mathrm{d}\alpha  \nonumber  \\
     & \ll & \max_{\alpha\in\mathfrak{m}}|f_2(\alpha)|^2
             \bigg(\int_{0}^1|f_2(\alpha)|^4\mathrm{d}\alpha\bigg)^{1/4} \bigg(\int_{0}^1|f_k(\alpha)|^4\mathrm{d}\alpha\bigg)^{1/4}
             \bigg(\int_{0}^1|f(-\alpha)|^2\mathrm{d}\alpha\bigg)^{1/2}    \nonumber  \\
     & \ll & x^{1+\varepsilon}Q^{-1}\cdot x^{3/4+1/(2k)+\varepsilon} \ll x^{7/4+1/(2k)+\varepsilon}Q^{-1}.
\end{eqnarray}

  \textbf{Case 2.} If~$k=6,$~from Lemma~\ref{Weyl-inequality} we have
\begin{eqnarray}
   f_6(\alpha) & \ll & x^{1/6+\varepsilon}\big(Q^{-1/32}+x^{-1/192}+\tau^{1/32}x^{-1/32}\big)   \nonumber   \\
               & \ll & x^{1/6+\varepsilon}\big(Q^{-1/32}+x^{-1/192}\big) .
\end{eqnarray}
Therefore, it follows from Cauchy's inequality, Lemma~\ref{Hua-inequality} and (\ref{f-a}) that
\begin{eqnarray}
     &     &  \mathcal{S}_6(x,\mathfrak{m})  =  \int_{\mathfrak{m}}f_2^3(\alpha)f_{6}(\alpha)f(-\alpha)\mathrm{d}\alpha  \nonumber  \\
     & \ll & \max_{\alpha\in\mathfrak{m}}|f_2(\alpha)|\cdot \max_{\alpha\in\mathfrak{m}}|f_6(\alpha)|
             \bigg(\int_{0}^1|f_2(\alpha)|^4\mathrm{d}\alpha\bigg)^{1/2}
             \bigg(\int_{0}^1|f(-\alpha)|^2\mathrm{d}\alpha\bigg)^{1/2}    \nonumber  \\
     & \ll & x^{5/3+\varepsilon}Q^{-17/32}+x^{319/192+\varepsilon}Q^{-1/2}.
\end{eqnarray}

 \textbf{Case 3.} If~$k=7,$~from Lemma~\ref{Weyl-inequality} we have
\begin{eqnarray}
   f_7(\alpha) & \ll & x^{1/7+\varepsilon}\big(Q^{-1/64}+x^{-1/448}+\tau^{1/64}x^{-1/64}\big)   \nonumber   \\
               & \ll & x^{1/7+\varepsilon}\big(Q^{-1/64}+x^{-1/448}\big) .
\end{eqnarray}
Therefore, it follows from Cauchy's inequality, Lemma~\ref{Hua-inequality} and (\ref{f-a}) that
\begin{eqnarray}
     &     &  \mathcal{S}_7(x,\mathfrak{m})  =  \int_{\mathfrak{m}}f_2^3(\alpha)f_{7}(\alpha)f(-\alpha)\mathrm{d}\alpha  \nonumber  \\
     & \ll & \max_{\alpha\in\mathfrak{m}}|f_2(\alpha)|\cdot \max_{\alpha\in\mathfrak{m}}|f_7(\alpha)|
             \bigg(\int_{0}^1|f_2(\alpha)|^4\mathrm{d}\alpha\bigg)^{1/2}
             \bigg(\int_{0}^1|f(-\alpha)|^2\mathrm{d}\alpha\bigg)^{1/2}    \nonumber  \\
     & \ll & x^{23/14+\varepsilon}Q^{-33/64}+x^{735/448+\varepsilon}Q^{-1/2}.
\end{eqnarray}

\textbf{Case 4.} If~$k\geqslant8,$~from Lemma~\ref{wooley-estimate} we have
\begin{eqnarray}
   f_k(\alpha) & \ll & x^{1/k+\varepsilon}\big(Q^{-1}+x^{-1/k}+\tau x^{-1}\big)^{1/(2k(k-1))}   \nonumber   \\
               & \ll & x^{1/k+\varepsilon}\big(Q^{-1/(2k(k-1))}+x^{-1/(2k^2(k-1))}\big) .
\end{eqnarray}
Therefore, it follows from H\"{o}lder's inequality, Lemma~\ref{Hua-inequality} and (\ref{f-a}) that
\begin{eqnarray}
     &     &  \mathcal{S}_k(x,\mathfrak{m})  =  \int_{\mathfrak{m}}f_2^3(\alpha)f_{k}(\alpha)f(-\alpha)\mathrm{d}\alpha  \nonumber  \\
     & \ll & \max_{\alpha\in\mathfrak{m}}|f_2(\alpha)|\cdot \max_{\alpha\in\mathfrak{m}}|f_k(\alpha)|
             \bigg(\int_{0}^1|f_2(\alpha)|^4\mathrm{d}\alpha\bigg)^{1/2}
             \bigg(\int_{0}^1|f(-\alpha)|^2\mathrm{d}\alpha\bigg)^{1/2}    \nonumber  \\
     & \ll & x^{3/2+1/k+\varepsilon}Q^{-(k^2-k+1)/(2k(k-1))}+x^{3/2+1/k-1/(2k^2(k-1))+\varepsilon}Q^{-1/2}.
\end{eqnarray}

 The rest thing that we need to do is to choose optimal parameters~$\tau$~and~$Q.$~By noting the condition (\ref{parameter}), we can
 use~$xQ^{-1}$~to substitute~$\tau$~in (\ref{major-arc-last}). Then, by a simple  calculation, it is easy to use Lemma~\ref{compare} to obtain the desired
 asymptotic formula of~$\mathcal{S}_k(x).$~This completes the proof of Theorem~\ref{mixed-power-theorem}.

\bigskip
\bigskip

\textbf{Acknowledgement}

   The authors would like to express the most and the greatest sincere gratitude to Professor Wenguang Zhai for his valuable
advice and constant encouragement.


\begin{thebibliography}{99}
  \bibitem{Calderon-Velasco}C. Calder\'{a}n, M. J. de Velasco, \textit{On divisors of a quadratic form}, Bol. Soc. Brasil. Mat., \textbf{31} (2000) 81-91.

  \bibitem{Gafurov-1}N. Gafurov, \textit{On the sum of the number of divisors of a quadratic form}, Dokl. Akad. Nauk Tadzhik., \textbf{28} (1985), 371-375.

  \bibitem{Gafurov-2}N. Gafurov, \textit{On the number of divisors of a quadratic form}, Proc. Steklov Inst. Math., \textbf{200} (1993) 137-148.


  \bibitem{Graham-Kolesnik}S. W. Graham, G. Kolesnik, \textit{Van der Corput's Method of Exponential Sums}, Cambridge University Press, 1991.

  \bibitem{Guo-Zhai}R. T. Guo,  W. G. Zhai, \textit{Some problems about the ternary quadratic form~$m_1^2+m_2^2+m_3^2$},
                              Acta Arith., \textbf{156} (2) (2012), 101-121.

  \bibitem{Hu}L. Hu, \textit{An asymptotic formula related to the divisors of the quaternary quadratic form}, Acta Arith., \textbf{166} (2014) 129-140.

  \bibitem{Liu-Hu}H. Liu, L. Hu, \textit{On the number of divisors of a quaternary quadratic form}, Int. J. Number Theory, 
                                \textbf{12} (05) (2016) 1219-1235.

  \bibitem{Lv-Mu}X. D. L\"{u}, Q. W. Mu, \textit{The Sum of Divisors of Mixed Powers}, Adv. Math. (China), \textbf{45} (3) (2016), 357-364.


  \bibitem{Pan-Pan}C. D. Pan, C. B. Pan, \textit{Goldbach Conjecture}, Science Press, Beijing, 1992.


  \bibitem{Srinivasan}B. R. Srinivasan, \textit{The lattice point problem of many-dimensional hyperboloids II}, Acta Arith., \textbf{8} (2) (1962), 173-204.

  \bibitem{Vaughan}R. C. Vaughan, \textit{The Hardy-Littlewood Method}, 2nd edn., Cambridge Tracts Math., Vol. 125, Cambridge University Press, 1997.

  \bibitem{Wooley}T. D. Wooley, \textit{Vinogradov's mean value theorem via efficient congruencing}, Ann. of Math. (2), \textbf{175} (3) (2012), 1575-1627.

  \bibitem{Yu}G. Yu, \textit{On the number of divisors of the quadratic form}, Canadian Math. Bull., \textbf{43} (2000) 239-256.

  \bibitem{Zhao}L. Zhao, \textit{The sum of divisors of a quadratic form}, Acta Arith., \textbf{163} (2014) 161-177.







\end{thebibliography}
\end{document}